\documentclass{amsart}
\usepackage[utf8]{inputenc}

\usepackage{amsmath}
\usepackage{amsfonts}
\usepackage{amssymb}
\usepackage{amsthm}
\usepackage{amscd}
\usepackage{graphicx}
\usepackage{hyperref}
\usepackage{color}
\usepackage{xcolor}
\usepackage[utf8]{inputenc}
\usepackage{mathtools}
\usepackage{fullpage}

\theoremstyle{plain}
\newtheorem*{theorem*}{Theorem}
\newtheorem{theorem}{Theorem}[section]

\newtheorem*{corollary*}{Corollary}

\newtheorem{proposition}{Proposition}
\newtheorem*{proposition*}{Proposition}
\newtheorem{lemma}[theorem]{Lemma}
\newtheorem*{lemma*}{Lemma}

\theoremstyle{definition}

\newtheorem*{definition*}{Definition}

\newtheorem*{example*}{Example}

\theoremstyle{remark}

\newtheorem*{remark*}{Remark}

\newcommand{\sC}{\mathcal{C}}
\newcommand{\sD}{\mathcal{D}}

\newcommand{\sI}{\mathcal{I}}
\newcommand{\sO}{\mathcal{O}}
\newcommand{\sF}{\mathcal{F}}

\newcommand{\sJ}{\mathcal{J}}
\newcommand{\sE}{\mathcal{E}}
\newcommand{\sL}{\mathcal{L}}

\newcommand{\sK}{\mathcal{K}}

\usepackage{tikz-cd}

\title{Stratified bundles on the Hilbert Scheme of $n$ points}
\date{}
\author{Saurav Holme Choudhury}
\address{The Institute of Mathematical Sciences, (HBNI), Chennai 600113.} 
\email{sauravhc@imsc.res.in}

\begin{document}

\maketitle

\begin{abstract}
Let $k$ be an algebraically closed field of characteristic $p > 3$ and $S$ be a smooth projective surface over $k$ with $k$-rational point $x$. For $n \geq 2$, let $S^{[n]}$ denote the Hilbert scheme of $n$ points on $S$. In this note, we compute the fundamental group scheme $\pi^{\textrm{alg}}(S^{[n]}, \tilde{nx})$ defined by the Tannakian category of stratified bundles on $S^{[n]}$.
\end{abstract}

\section{Introduction}

For a variety $X$ over $\mathbb{C}$, one has the classical notion of the fundamental group $\pi_1(X^{\textrm{an}},x)$ defined using the analytic topology on $X$. Over arbitrary base fields $k$, one has several analogues of the fundamental group defined in terms of algebro-geometric information.

In \cite{SGA1}, Grothendieck introduced the notion of \'etale fundamental group $\pi^{\textrm{\'et}}(X, x)$, where $X$ is a scheme and $x$ is a geometric point of $X$, in terms of the finite etale covers of $X$. In \cite{N76}, Nori defined the Nori fundamental group scheme $\pi^N(X, x)$, where $X$ is a connected, reduced and complete scheme over a perfect field $k$ and $x$ is a $k$-rational point, via Tannakian reconstruction using the category of essentially finite vector bundles on $X$. The definition of $\pi^N(X, x)$ was extended to the case of connected and reduced $k$-schemes in \cite{N82}. Another analogue, the S-fundamental group scheme $\pi^S(X, x)$ was introduced and studied by Langer in \cite{Lan11} and \cite{Lan12} for smooth projective varieties $X$ over an algebraically closed field $k$. It is defined via Tannakian reconstruction using the category of numerically flat vector bundles on $X$. The S-fundamental group scheme for a smooth projective curve $C$ over an algebraically closed field $k$ was already introduced and studied in \cite{BPS06}.

The variant of the fundamental group scheme which is of prime importance in this note is the algebraic fundamental group $\pi^{\textrm{alg}}(X, x)$. In \cite{Gie75}, Gieseker defined $\pi^{\textrm{alg}}(X, x)$ as  the fundamental group scheme corresponding to the Tannakian category of $\sD_X$-modules, where $\sD_X$ is the sheaf of differential operators on $X$. For $X$ smooth over a field of positive characteristic, Gieseker introduced the notion of \textit{stratified bundles} and showed that the category of $\sD_X$-modules is tensor equivalent to the category of stratified bundles on $X$. Stratifed bundles were further studied in  \cite{dS07} and \cite{BHdS21}. Precise definitions and statements will be given in the next section.
\\

Let $S$ be a smooth projective surface over $k$. For $n\geq 2$, let $S^{[n]}$ denote the Hilbert scheme $n$ points on $S$. It is well known that $S^{[n]}$ is a smooth projective variety of dimension $2n$. In \cite{PS20}, the authors show that for char $k >3$ and $n\geq 2$, there is an isomorphism of affine group schemes over $k$

$$\pi^{\dagger}(S, x)_{\textrm{ab}}\to \pi^{\dagger}(S^{[n]}, \tilde{nx})$$

where $\dagger = S, N \textrm{ or \'et}$.
\\

In this note, we extend their results to the case of $\pi^{\textrm{alg}}$ and prove the following theorem.

\begin{theorem*}
Let char $k > 3$ and $n \geq 2$. There is an isomorphism of affine group schemes over $k$ 

$$f:  \pi^{\textrm{alg}}(S, x)_{\textrm{ab}} \to  \pi^{\textrm{alg}}(S^{[n]}, \tilde{nx})$$ 
\end{theorem*}

In section \ref{strat}, we recall the definition of stratified bundles and some of their basic properties. The formalism of Tannakian reconstruction is recalled in section \ref{tanngrp} and used to define the algebraic fundamental group $\pi^{\textrm{alg}}(X, x)$.

\pagebreak

The geometrical properties of the Hilbert scheme of $n$ points on a smooth projective surface are in section \ref{hsnpt}. In section \ref{functor}, we prove a result about descent of stratified bundles which allows us to define the homomorphism $f$ by defining the associated functor of Tannakian categories. The concluding section \ref{iso} establishes the main theorem by showing that $f$ is an isomorphism. 

\subsection*{Acknowledgements} We would like to thank Indranil Biswas and Ronnie Sebastian for their comments on earlier drafts of this note.

\section{Stratified bundles}\label{strat}

Let $k$ be a field of characteristic $p$ and $X$ be a noetherian scheme over $k$. Stratified bundles on $X$ are sequences of coherent sheaves on $X$ satisfying infinite Frobenius descent. More precisely, the category of stratified bundles on $X$, denoted $\mathcal{S}(X)$, consists of 

\begin{itemize}
\item \textbf{Objects} $(\sE_i, \alpha_i)$ are sequences of coherent $\sO_X$-modules $\sE_i$, $i\in \mathbb{N}$ along with isomorphisms

$$\alpha_i:F^*\sE_{i+1}\to \sE_i$$

for all $i \in \mathbb{N}$, where $F$ is the absolute Frobenius on $X$.

\item \textbf{Morphisms} $\phi: (\sE_i, \alpha_i) \to (\sF_i, \beta_i)$ consists of a sequence of $\sO_X$-module morphisms $\phi_i: \sE_i \to \sF_i$ such that $\phi_i \circ \alpha_i = \beta_i \circ F*(\phi_{i+1})$ 
\end{itemize}

Let $f: Y \to X$ be a morphism and $(\sE_i, \alpha_i)$ be a stratified bundle on $X$. Then we can define the pullback along $f$, denoted $f^*(\sE_i, \alpha_i)$, as consisting of the sequence of $O_Y$ coherent sheaves $f^*E_i$ and isomorphisms are given by the composite maps

$$F^*f^*\sE_{i+1} \xrightarrow{\gamma_{\sE_{i+1}}} f^*F^*\sE_{i+1} \xrightarrow{f^*(\alpha_i)} f^*\sE_i$$

where $\gamma : F^*f^* \to f^*F^*$ is the natural isomorphism of functors.
\\

Thus $\mathcal{S}(X)$ is contravariant functor in $X$. One also has a tensor product on $\mathcal{S}(X)$ defined by taking term by term tensor product. Also $\mathcal{S}(X)$ is an abelian category [cf. \cite{BHdS21}, Proposition 4.4].
\\

We recall some well known results about stratified bundles [cf \cite{dS07}, \cite{Gie75}].

\begin{proposition*}
If $(\sE_i, \alpha_i)$ is a stratified bundle on $X$, then $\sE_i$ is a locally free $\sO_X$-module for all $i \in \mathbb{N}$.
\end{proposition*}

This allows us to define duals of stratified bundles, making $\mathcal{S}(X)$ into an \textit{abelian rigid tensor category}.
\\

The \textit{rank} of a stratified bundle $(\sE_i, \alpha_i)$ is defined to be the rank of $\sE_0$. The \textit{trivial stratified bundles} on $X$ are of the form $\oplus (\sO_X,...; F^*,...)$.
\\

Let $\sD_X$ be the sheaf of differential operators on $X$. The category of $\sD_X$ modules consists of

\begin{itemize}
\item \textbf{Objects} coherent $\sO_X$ modules $\sE$ equipped with a 
$\sD_X$ action i.e a morphism of $\sO_X$-algebras

$$\sD_X \to \sE\textrm{nd}_k(\sE)$$
\item \textbf{Morphisms}  $\sO_X$-linear maps $\sE \to \sF$ compatible with the $\sD_X$ action
\end{itemize} 

A theorem of Katz [\cite{Gie75}, Theorem 1.3] shows that for $X$ smooth over $k$, then the category of stratified bundles on $X$ and the category of $\sD_X$ modules are tensor equivalent to each other.
\\

We close this section with the definition of $G$ equivariant stratified bundles on a variety $X$ admitting action of a group $G$ on it.

\begin{definition*}
A stratified bundle $(\sE_i, \alpha_i)$ is said to be a $G$-equivariant stratified bundle if $\sE_i$ are $G$-equivariant vector bundles and $\alpha_i$ are $G$-equivariant $\sO_X$ module morphisms.
\end{definition*}

\section{Tannakian categories and fundamental group schemes}\label{tanngrp}

In this section we recall the definition and basic properties of Tannakian categories. We then recall Gieseker's definition of the fundamental group scheme $\pi^{\textrm{alg}}$ using the Tannakian formalism.

\subsection{Tannakian Categories and affine group schemes}

Tannakian categories were defined and studied in \cite{DM82} to formalize the properties of $\textrm{Rep}_k(G)$, the category of finite dimensional $k$-representations of $G$, an affine group scheme over $k$.

\begin{definition*}[Neutral Tannakian Categories]
A rigid abelian tensor category $\sC$ with End $\mathbb{I} = k$ is a \textit{neutral Tannakian category} if it admits an exact faithful $k$-linear tensor functor $\omega: \sC \to \textrm{Vec}_k$. Any such functor is said to be a \textit{fiber functor} for $\sC$.
\end{definition*}

Given a neutral Tannakian category $(\sC, \otimes, \omega, \mathbb{I})$, we define the functor $\textrm{Aut}^{\otimes}(\omega): k-\textrm{algebra} \to \textrm{Sets}$ such that for $k$-algebra $R$, $\textrm{Aut}^{\otimes}(\omega)(R)$ consists of the families $(\lambda_X)$ for $X \in \textrm{ob}(\sC)$, where $\lambda_X$ is a $R$-linear automorphism of $X \otimes R$ such that $\lambda_{X_1 \otimes X_2} = \lambda_{X_1}\otimes \lambda_{X_2}$, $\lambda_{\mathbb{I}} = id_R$, and 

$$\lambda_Y \circ (\alpha \otimes 1) = (\alpha \otimes 1) \circ \lambda_X : X\otimes R \to Y \otimes R$$

for all morphisms $\alpha: X \to Y$.

\begin{theorem*}
[Main theorem for neutral Tannakian categories, \cite{DM82}, Theorem 2.11] Let $(\mathcal{C}, \otimes)$ be a rigid abelian tensor category such that $k = \textrm{End}(\mathbb{I})$ and let $\omega: \mathcal{C} \to \textrm{Vec}_k$ be an exact faithful tensor functor. Then

\begin{itemize}
\item The functor $\textrm{Aut}^{\otimes}(\omega)$ of $k$-algebras is represented by an affine group scheme $G$.
\item The functor $\mathcal{C}\to \textrm{Rep}_k(G)$ is an equivalence of tensor categories.
\end{itemize}
\end{theorem*}

\begin{theorem*}
Let $(\sC, \otimes, \omega, \mathbb{I})$ and $(\sC', \otimes, \omega', \mathbb{I'})$ be neutral Tannakian categories which correspond to the representation categories of the affine $k$ group schemes $G$ and $G'$ respectively. Then any functor of Tannakian categories from $\sC \to \sC'$ is induced by a unique morphism of affine $k$ group schemes $G' \to G$.
\end{theorem*}

This theorem allows us to define many variants of fundamental groups of a scheme $X$ by considering different Tannakian categories naturally associated with $X$. The following result is very useful in establishing a given morphism between affine group schemes is an isomorphism.

\begin{theorem*}[\cite{DM82}, Theorem 2.21] Let $f: G \to G'$ be a homomorphism of group schemes over $k$ and $\textrm{Rep }(f):\textrm{Rep }(G') \to \textrm{Rep }(G)$ be the corresponding functor of Tannakian categories. Then

\begin{itemize}\label{ffci}
\item $f$ is faithfully flat if and only if $\textrm{Rep }(f)$ is fully faithful and has essential image closed under subobjects i.e
for $V' \in \textrm{Rep }(G')$ and suboject $W \subset \textrm{Rep }(f)(V')$, there is a subobject $W' \subset V'$ in $\textrm{Rep }(G')$ such that $\textrm{Rep }(f)(W') \simeq W$ in $\textrm{Rep }(G)$
\item $f$ is closed immersion if and only if every object of $\textrm{Rep }(G)$ is a subquotient of some object in the essential image of $\textrm{Rep }(f)$.
\end{itemize}

\end{theorem*}

We finish by recalling a basic result on affine group schemes (we refer to section $4.1$ in \cite{PS20} for details).
\\

Let $G$ be a affine group scheme over $k$, $G_{\textrm{ab}}$ be its abelianization (i.e the maximal abelian quotient of $G$) and $\alpha: G \to G_{\textrm{ab}}$ be the (faithfully flat) quotient morphism  . We can then define the composite morphism

$$\phi: G^n \xrightarrow{\alpha^n} G_{\textrm{ab}}^n \xrightarrow{m} G_{\textrm{ab}}$$

where $m$ is the multiplication homomorphism. As $\mathfrak{S}_n$ acts on the $k$-group scheme $G^n$, we can define the notion of a $\mathfrak{S}_n$-invariant group morphism $\psi: G^n \to H$ for any $k$-group scheme $H$.

\begin{lemma}\label{grpschlemm}
Let $G$ and $H$ be two group schemes over $k$. For an integer $n\geq 2$, the set of $\mathfrak{S}_n$-invariant group morphisms $G^n \to H$ is in bijective correspondence with the set of group morphism $G_{\textrm{ab}} \to H$ i.e any morphism of $k$-group schemes $\phi: G^n \to H$ which is $\mathfrak{S}_n$-invariant factors uniquely through a morphism $\psi: G_{\textrm{ab}} \to H$ such that $\phi = \psi\circ h$
\end{lemma}

\subsection{The group scheme $ \pi^{\textrm{alg}}(X, x)$}

Classically, over $\mathbb{C}$, the Riemann-Hilbert correspondence identifies the category of vector bundles equipped with integrable connections on  a smooth connected projective variety $X$/$\mathbb{C}$ with the category of representations of the topological fundamental group $\pi^{\textrm{top}}(X,x)$ for some chosen base point $x$. Via GAGA, this gives a purely algebraic description of the category of representations of the topological fundamental group $\pi(X,x)$. This category (equipped with the fiber functor $(E, \nabla) \to E_x$) is a neutral Tannakian category and can be identified, via the Tannakian formalism, with the representation category of the proalgebraic completion of the topological fundamental group, denoted as $\pi^{\textrm{top}}(X,x)^{\textrm{alg}}$.
\\

Over a field $k$ of characteristic $0$, the category of flat connections on a smooth variety $X$ is tensor equivalent to the category of $\sD_X$-modules. However over a field of characteristic $p$, the category of flat connections on $X$ is not as well behaved as the category of $\sD_X$-modules and one defines a fundamental group scheme for $X$ by Tannakian formalism using the category of $\sD_X$-modules. By Katz's theorem mentioned before, the fundamental group coincides with the one defined using $\mathcal{S}(X)$ below.
\\

Let $x \in X(k)$ be a $k$-rational point. Then the abelian rigid tensor category $\mathcal{S}(X)$ is neutralized by the fiber functor

$$T_x: \mathcal{S}(X) \to Vec_k$$

The fundamental group scheme defined by the neutral Tannakian category $(\mathcal{S}(X), \otimes, T_x, (\sO_X, F^*))$ is called the algebraic fundamental group of $X$ based at $x$ and is denoted by $\pi^{\textrm{alg}}(X,x)$.
\\

The following basic properties of $\pi^{\textrm{alg}}$ are well known.

\begin{itemize}
\item (Independence of basepoint) Let $X$ be a geometrically connected, smooth projective $k$-scheme. Then for all $x_1, x_2 \in X(k)$, one has

$$\pi^{\textrm{alg}}(X, x_1) \simeq \pi^{\textrm{alg}}(X, x_2)$$

\item (Product rule) \label{product}
For $X_1, X_2$ geometrically connected and smooth over $k$ and $x_i \in X_i(k)$, there is an isomorphism

$$\pi^{\textrm{alg}}(X_1, x_1) \times \pi^{\textrm{alg}}(X_2, x_2) \to \pi^{\textrm{alg}}(X_1\times X_2, (x_1, x_2))$$

\item For $X$ smooth and open immersion $U \xrightarrow{i} X$  such that the complement of $U$ in $X$ has codimension $\geq 2$ and $x \in U(k)$, then the homomorphism $$\pi^{\textrm{alg}}(U, x) \to \pi^{\textrm{alg}}(X, x)$$ associated to the restriction functor $i^*: \mathcal{S}(X) \to \mathcal{S}(U)$ is an isomorphism. 
\end{itemize}

\section{Geometry of Hilbert Scheme of points}\label{hsnpt}
 
Let $S$ be a smooth projective surface over $k$. We fix notation as follows

\begin{itemize}
    \item $S^n$ denotes the $n$-fold cartesian product of $S$ with itself.
    \item $S^{(n)}$ denotes the $n$th symmetric product of $S$ defined as the quotient $S^n/\mathfrak{S}_n$, where $\mathfrak{S}_n$ denotes the symmetric group on $n$ letters.
    \item $S^{[n]}$ denotes the Hilbert scheme of $n$ points on $S$.
\end{itemize}

    Let $\rho: S^n \to S^{(n)}$ be the quotient map and $h: S^{[n]} \to S^{(n)}$ be the Hilbert-Chow morphism. We write $S^{(n)}_{\circ}$ for the open subset of $S^{(n)}$ consisting of distinct points with $S^{[n]}_{\circ} := h^{-1}(S^{(n)}_{\circ})$ and $S^{n}_{\circ} := \rho^{-1}(S^{(n)}_{\circ})$. The map $h_{n, \circ}: S^{[n]}_{\circ} \to S^{(n)}_{\circ}$ is an isomorphism. We have the diagram:
    
    \[\begin{tikzcd}
	{S^{[n]}} && {S^n} \\
	\\
	&& {S^{(n)}}
	\arrow["{h_{n}}", from=1-1, to=3-3]
	\arrow["{\rho_{n}}"', from=1-3, to=3-3]
\end{tikzcd}\]

In general, Hilbert schemes of points on a projective variety display a lot of pathological features. But in \cite{Fog68} the author shows that, in the case of smooth projective surface $S$, $S^{[n]}$ is a smooth projective variety. Thus, in this case, the Hilbert-Chow morphism $h: S^{[n]} \to S^{(n)}$ is a resolution of singularities. 

\

One can consider $S^{(n)}$ as the set of effective $0$-cycles of degree $n$ on $S^{(n)}$. In this case it is easy to see that $S^{(n)}$ admits a stratification by \textit{type}, where the type of a $0$-cycle $y$ of degree $n$ is a tuple $(n_1,\dots, n_r)$ where $y$ can be written as 

$$y = \Sigma_{j=1}^r n_jx_j$$

where $x_j$ are distinct points of $S$ with multiplicities $n_1 \geq n_2 \geq \dots \geq n_r$, where $n_j$ are positive integers. 

Let $C(n_1,\dots, n_r)$ denote the subset of $S^{(n)}$ consisting of points of the type $(n_1,\dots, n_r)$. Let $S^{(n)}_* = C(1,\dots,1) \bigcup C(2,\dots, 1)$ denote the open subset of $S^{(n)}$ consisting of points of type $(1,\dots,1)$ and $(2,\dots,1)$. Let $S^{[n]}_*$ and $S^{n}_*$ denote the preimage of $S^{(n)}_*$ under $h$ and $\rho$ respectively.
\\

We recall some basic properties below which we will need later (we refer to \cite{Fog68}, \cite{PS20} for details).

\begin{itemize}
\item The subsets $C(n_1,\dots, n_r)$ are nonsingular of dimension $2r$.
\item The closed subset $S^{(n)} \setminus S^{(n)}_*$ is of codimension $\geq 2$ in $S^{(n)}$.
\item The closed subset $S^{[n]} \setminus S^{[n]}_*$ is of codimension $2$ in $S^{[n]}$.
\item The closed subset $S^n \setminus S^n_*$ is of codimension $\geq 4$ in $S^n$.
\item The closed subset $S^{(n)}_*\setminus S^{(n)}_{\circ}$ is of codimension $2$ in $S^{(n)}_*$.
\item When characteristic of $k \neq 2$, for $y \in C(2,1,\dots,1)$, the scheme theoretic fiber $h^{-1}(y)$ is isomorphic to $\mathbb{P}^1_k$. In fact, $S^{[n]}_*$ is the blowup of $S^{(n)}_*$ along $C(2,1,\dots,1)$.
\end{itemize}

We end this section by recalling a result of Fogarty (\cite{Fog77}, Proposition 3.6).

\begin{proposition*}
If $L$ is a $\mathfrak{S}_n$-invariant line bundle on $S^n$, there exists a line bundle $L'$ on $S^{(n)}$ such that $h^*L' \simeq L$.
\end{proposition*}

It follows that $L'$ in the proposition is isomorphic to $\sigma_*(L)^{\mathfrak{S}_n}$

\section{The functor between Tannakian categories}\label{functor}

Let $S$ be a smooth projective surface over $k$ and $(\sE_i, \alpha_i)$ be a stratified bundle on $S^{[n]}$. Restricting to $S^{[n]}_*$ gives us a functor 

$$i^*: \mathcal{S}(S^{[n]}) \to \mathcal{S}(S^{[n]}_*)$$

which is a equivalence of categories as $S^{[n]}_*$ is the complement of a codimension $2$ closed subset of $S^{[n]}$.
\\

Next we show that a stratified bundle on $S^{[n]}_*$ can be pushed forward under $h$ to get a stratified bundle on $S^{(n)}_*$. First we begin by a result on descent of vector bundles along the morphism $h: S^{[n]}_* \to S^{(n)}_*$. Similar results have been established by authors in \cite{Ish83} and \cite{PS20}.

\begin{proposition}
Assume char $k \neq 2$. Let $\sE$ be a vector bundle on $S^{[n]}_*$ which restricts to trivial vector bundles on the fibers of $h$ over $S^{(n)}_*$. Then $h_*\sE$ is a locally free $\sO_{S^{(n)}_*}$-module. Moreover the natural map

$$h^*h_*(\sE) \to \sE$$

is an isomorphism.
\end{proposition}

\begin{proof}
Let $x \in S^{(n)}_*$ be a point of type $(2,1,\dots , 1)$. Then by assumption, the fiber of $h$ over $x$ is isomorphic to $\mathbb{P}^1_k$. Let $\mathcal{J}$ be the ideal sheaf of the closed subscheme $h^{-1}(x)$ and $\sI_x$ be the ideal sheaf of the closed point $x$. We have 

$$\mathcal{J} = \sI_x\sO_{S^{[n]}_*}$$

For all $n \geq 1$, let $Y_n$ denote the closed subscheme of $S^{[n]}_*$ corresponding to the ideal sheaf $\mathcal{J}^n$. Consider the following short exact sequence of sheaves on $S^{[n]}_*$

$$0 \to \sJ\otimes\sE \to \sE \to \sE|_{Y_1} \to 0$$

Pushing forward by $h$, we get the following exact sequence of sheaves on $S^{(n)}_*$

$$h_*\sE \to H^0(Y_1, \sE|_{Y_1}) \to R^1h_*(\sJ\otimes \sE)$$

We claim that the completion of $R^1h_*(\sJ\otimes \sE)$ at the maximal ideal $m_x$ in $\sO_{S^{(n)}_*, x}$ is $0$. The proof uses the theorem of formal functions which says that 

$$(R^1h_*(\sJ\otimes \sE))^{\wedge} \simeq \varprojlim H^1(Y_n, \sJ\otimes \sE \otimes \sO_{S^{[n]}_*}/\sJ^n)$$

We prove by induction that $H^1(Y_n, \sJ\otimes \sE \otimes \sO_{S^{[n]}_*}/\sJ^n) = 0$. As $Y_1 \simeq \mathbb{P}^1_k$, the sheaves $\sJ^n/\sJ^{n+1}$ are locally free. These sheaves are also globally generated over $Y_1$ as we have the surjection 

$$m_x^n/m_x^{n+1}\otimes_{\sO_{S^{(n)}_*, x}} \sO_{S^{[n]}_*} \simeq \sI_x^n/\sI_x^{n+1}\otimes_{\sO_{S^{(n)}_*}} \sO_{S^{[n]}_*} \twoheadrightarrow \sJ^n/\sJ^{n+1}$$

As $\sJ^n/\sJ^{n+1}$ is locally free on $Y_1 \simeq \mathbb{P}^1_k$ and globally generated, it is a direct sum of line bundles each of which has degree $\geq 0$. Thus one gets the base case of induction from degree considerations, as

$$H^1(Y_1, \sJ\otimes \sE \otimes \sO_{S^{[n]}_*}/\sJ = H^1(Y_1, \sJ/\sJ^2 \otimes \sE_{Y_1}) = 0$$

Assume that the claim is true for $n$. Then the proof for $n+1$ follows from the long exact sequence in cohomology attached to the short exact sequence of sheaves on $Y_{n+1}$

$$0 \to \sJ^{n+1}/\sJ^{n+2}\otimes \sE \to \sJ/\sJ^{n+2}\otimes \sE \to \sJ/\sJ^{n+1}\otimes \sE \to 0$$

which gives us the exact sequence

$$ H^1 (Y_{n+1}, \sJ^{n+1}/\sJ^{n+2}\otimes \sE) \to H^1(Y_{n+1}, \sJ/\sJ^{n+2}\otimes \sE) \to H^1(Y_{n+1} , \sJ/\sJ^{n+1}\otimes \sE)$$

We know $H^1 (Y_n, \sJ^{n+1}/\sJ^{n+2}\otimes \sE) = H^1 (Y_1, \sJ^{n+1}/\sJ^{n+2}\otimes \sE) = 0$ (by degree consideration) and $H^1(Y_{n+1}, \sJ/\sJ^{n+1}\otimes \sE) = H^1(Y_n, \sJ/\sJ^{n+1}\otimes \sE) = 0$ (by induction hypothesis), thus we get 

$$ H^1(Y_{n+1}, \sJ/\sJ^{n+2}\otimes \sE) = 0$$

Thus the stalk of $R^1h_*(\sJ\otimes \sE)$ at $x$ is $0$.
\\

This shows that the natural map $h_*\sE \to H^0(Y_1, \sE|_{Y_1})$ is surjective in a neighbourhood of $x$. Let $f_1,...,f_r$ be a basis of $H^0(Y_1, \sE|_{Y_1})$. Let $\textrm{Spec}(R)$ be an affine neighbourhood of $x$ where the natural map is surjective and let $\tilde{f_i} \in \Gamma(\textrm{Spec}(R), h_*\sE) = \Gamma (h^{-1}(\textrm{Spec}(R)), \sE)$ be lifts of $f_i$. Using $\tilde{f}_i$ one defines a homomorphism

$$\sO_{S^{[n]}_*}^{\oplus r}|_{h^{-1}(\textrm{Spec}(R))} \to \sE$$

on $h^{-1}(\textrm{Spec}(R)$ which is a surjection (and hence an isomorphism) on $Y_1$. As $h$ is proper, there exists a smaller affine neighbourhood $U$ of $x$ over which there is an isomorphism

$$\sO_V^{\oplus r} \simeq \sE$$

where $V = h^{-1}(U)$. Applying $h_*$, we get

$$(h_*\sO_V)^{\oplus r} \simeq h_*\sE $$

As $S^{(n)}_*$ is normal and $h: S^{[n]}_* \to S^{(n)}_*$ is birational with connected fibers, by a form of Zariski's main theorem [cf \cite{Har77}, Corollary 11.3 and 11.4], we have that $h_*\sO_V \simeq \sO_U$ and thus $h_*\sE$ is locally free. The natural morphism

$$h^*h_*(\sE) \to \sE$$

is clearly an isomorphism.

\end{proof} 

Let $\mathbb{VB}_{S^{(n)}_*}$ be the category of locally free sheaves on $S^{(n)}_*$ and $\mathbb{VB}^h_{S^{[n]}_*}$ be the category of locally free sheaves on $S^{[n]}_*$ which restrict to trivial vector bundles on the fibers of $h$. Proposition 1 above gives us an equivalence of categories.

\begin{proposition}
Assume char $k \neq 2$. The pushforward functor 
$$h_*: \mathbb{VB}^h_{S^{[n]}_*} \to \mathbb{VB}_{S^{(n)}_*}$$ 

is an equivalence of categories with the inverse given by 

$$h^*: \mathbb{VB}_{S^{(n)}_*} \to  \mathbb{VB}^h_{S^{[n]}_*}$$
\end{proposition}

\begin{proof} 
We observe that if $\sE' \simeq h^*(\sE)$, then $\sE \simeq h_* \sE'$. This shows that $h_*$ is essentially surjective. The natural map

$$\textrm{Hom}_{S^{(n)}_*}(h_* \sE, h_* \sF) \to \textrm{Hom}_{S^{[n]}_*}(\sE, \sF)$$

is bijective. Thus $h_*$ is an equivalence of categories. 
\end{proof}

\begin{corollary*}
For all $\sE \in \mathbb{VB}^h_{S^{[n]}_*}$, the natural map $$F^*h_*(\sE) \to h_*F^*(\sE)$$ is an isomorphism over $S^{(n)}_*$.
\end{corollary*}

\begin{proof}
As $F^*\sE$ is also an object of $\mathbb{VB}^h_{S^{[n]}_*}$, thus both sheaves are locally free of the same rank. Thus it suffices to show that the natural map

$$F^*h_*(\sE) \to h_*F^*(\sE)$$

is surjective. As $F$ is faithfully flat on the smooth locus of $S^{(n)}_*$, the claim holds on the smooth locus. Let $x \in S^{(n)}_*$ be of type $(2,1,\dots, 1)$. Then the restriction of $F^*h_*(\sE)$ to $x$ is naturally isomorphic to $H^0(Y_1, \sE|_{Y_1})$ and the restriction of $ h_*F^*(\sE)$ to $x$ is $H^0(Y_1, F^*(\sE|_{Y_1})$. The restriction of the natural map to $x$ is the map 

$$F^*: H^0(Y_1, \sE_1) \to H^0(Y_1, F^*\sE_1)$$

which is surjective. 
\end{proof}

By Theorem $2.2$ of \cite{Gie75}, we have that every stratified bundle on $\mathbb{P}^1_k$ is trivial. Thus the above results give us

\begin{proposition}
Assume char $k \neq 2$. Let $(\sE_i, \alpha_i)$ be a stratified bundle on $S^{[n]}_*$. Then $h_*(\sE_i)$ is locally free $\sO_{S^{(n)}_*}$-module for all $i \in \mathbb{N}$. Moreover the natural map

$$h^*h_*(\sE_i) \to \sE_i$$

is an isomorphism. Furthermore the natural map $$F^*h_*(\sE_i) \to h_*F^*(\sE_i)$$ is an isomorphism over $S^{(n)}_*$.
\end{proposition}

This allows us to define the pushforward of a stratified bundle $(\sE_i, \alpha_i)$ on $S^{[n]}_*$. The pushforward denoted $h_*(\sE_i, \alpha_i)$ is given by the sequence of vector bundles $h_*\sE_i$ for all $i \in \mathbb{N}$ and the isomorphisms are given by the composite

$$F^*h_*(\sE_{i+1}) \xrightarrow{\eta_{\sE_{i+1}}} h_*F^*(\sE_{i+1}) \xrightarrow{h_*(\alpha_i)} h_*(\sE_i)$$

where $\eta: F^*h_* \to h_*F^*$ is the natural transformation.
\\

\pagebreak 

Thus we get a functor 

$$h_*: \mathcal{S}(S^{[n]}_*) \to \mathcal{S}(S^{(n)}_*)$$

$h_*$ is additive tensor functor as on the smooth locus $S^{(n)}_{\circ}$ we have the isomorphisms

$$h_*((\sE_i, \alpha_i) \oplus (\sF_i, \beta_i))|_{S^{(n)}_{\circ}} \simeq h_*(\sE_i, \alpha_i)|_{S^{(n)}_{\circ}} \oplus h_*(\sF_i, \beta_i))|_{S^{(n)}_{\circ}}$$

$$h_*((\sE_i, \alpha_i) \otimes (\sF_i, \beta_i))|_{S^{(n)}_{\circ}} \simeq h_*(\sE_i, \alpha_i)|_{S^{(n)}_{\circ}} \otimes h_*(\sF_i, \beta_i))|_{S^{(n)}_{\circ}}$$

which extend to $S^{(n)}_*$ due to codimension reasons.
\\

The following commutative diagram shows that $h^*h_*(\sE_i, \alpha_i)$ is isomorphic to $(\sE_i, \alpha_i)$ as stratified bundles with the isomorphism given by the natural morphisms $h^*h_*\sE_i \to \sE_i$.

\begin{center}
\begin{tikzcd}
	{F^*h^*h_*\mathcal{E}_{i+1}} && {h^*F^*h_*\mathcal{E}_{i+1}} && {h^*h_*F^*\mathcal{E}_{i+1}} && {h^*h_*\mathcal{E}_{i}} \\
	\\
	{F^*\mathcal{E}_{i+1}} &&&& {F^*\mathcal{E}_{i+1}} && {\mathcal{E}_i}
	\arrow["{h^*\eta_{\mathcal{E}_{i+1}}}", from=1-3, to=1-5]
	\arrow["{h_*h^*\alpha_i}", from=1-5, to=1-7]
	\arrow[from=1-7, to=3-7]
	\arrow[from=1-5, to=3-5]
	\arrow["{\alpha_i}", from=3-5, to=3-7]
	\arrow["{\gamma_{h_*\mathcal{E}_{i+1}}}", from=1-1, to=1-3]
	\arrow[from=1-1, to=3-1]
	\arrow["{=}", from=3-1, to=3-5]
\end{tikzcd}

\end{center}

Consider the pullback functor

$$\rho^* : \mathcal{S}(S^{(n)}_*) \to \mathcal{S}(S^{n}_*)$$

which takes values in the category of $\mathfrak{S}_n$-equivariant stratified bundles on $S^{n}_*$. Also we have the extension functor

$$j_*:\mathcal{S}(S^{n}_*) \to \mathcal{S}(S^n)$$

which is an equivalence of categories. Composing these functors together, we get a functor 

$$T: \mathcal{S}(S^{[n]}) \to \mathcal{S}(S^n)$$

given by $$T = j_*\circ \rho^*\circ h_*\circ i^*$$

Clearly $T$ is an additive tensor functor. Note that $h_*$ is fully faithful, $\rho^* : \mathcal{S}(S^{(n)}_*) \to \mathcal{S}(S^{n}_*)$ is fully faithful (as $\rho: S^n_{\circ} \to S^{(n)}_{\circ}$ is finite \'etale) and $j_*:\mathcal{S}(S^{n}_*) \to \mathcal{S}(S^n)$ is an equivalence of categories (due to codimension reasons). Thus $T$ is fully faithful.

\subsection{The homomorphism}

Fix $n$ distinct $k$-valued points $x_1,\dots , x_k \in S(k)$. Let $\tilde{x} \in S^{[n]}$ such that $h(\tilde{x}) = \sigma(x_1,\dots , x_n) = z \in S^{(n)}_{\circ}$. Then the categories $\mathcal{S}(S^{[n]})$ and  $\mathcal{S}(S^n)$ are neutralized by the respective fiber functors 

$$\tau_{\tilde{x}}: \mathcal{S}(S^{[n]}) \to \textrm{Vec}_k$$
$$(\sE_i, \alpha_i) \mapsto (\sE_0)_{\tilde{x}}$$

$$\tau_{(x_1,\dots , x_n)}: \mathcal{S}(S^n) \to \textrm{Vec}_k$$
$$(\sF_i, \beta_i) \mapsto (\sF_0)_{(x_1,\dots , x_n)}$$

If $T((\sE_i, \alpha_i)) = (\sF_i, \beta_i)$ that we have natural isomorphisms $(\sE_0)_{\tilde{x}} \simeq h_*(\sE_0)_{z}\simeq (\sF_0)_{(x_1,\dots , x_n)}$.
\\

Thus we have a functor of Tannakian categories 

$$T: (\mathcal{S}(S^{[n]}), \otimes, \tau_{\tilde{x}}, (\sO_{S^{[n]}}, d)) \to (\mathcal{S}(S^{n}), \otimes, \tau_{(x_1,\dots , x_n)}, (\sO_{S^{n}}, d))$$ 
\\

which by the independence of basepoint property of $\mathcal{S}$ induces a functor of Tannakian categories

$$T: (\mathcal{S}(S^{[n]}), \otimes, \tau_{\tilde{nx}}, (\sO_{S^{[n]}}, d)) \to (\mathcal{S}(S^{n}), \otimes, \tau_{(x,\dots , x)}, (\sO_{S^{n}}, d))$$ 

and hence a morphisms of the associated fundamental group schemes

$$\tilde{f}:  \pi^{\textrm{alg}}(S^n, (x,\dots , x)) \to  \pi^{\textrm{alg}}(S^{[n]}, \tilde{nx})$$

Note that by proposition \ref{product} we have

$$ \pi^{\textrm{alg}}(S^n, (x,\dots , x))\simeq  \pi^{\textrm{alg}}(S, x)^n$$.

As $$T: (\mathcal{S}(S^{[n]}), \otimes, T_{\tilde{nx}}, (\sO_{S^{[n]}}, d)) \to (\mathcal{S}(S^{n}), \otimes, T_{(x,\dots , x)}, (\sO_{S^{n}}, d))$$ 

takes stratified bundles on $S^{[n]}$ to $\mathfrak{S}_n$-equivariant stratified bundles on $S^n$ and a $\mathfrak{S}_n$-equivariant stratified bundles on $S^n$ corresponds to a $\mathfrak{S}_n$-invariant representation of $ \pi^{\textrm{alg}}(S, x)^n$, by \ref{grpschlemm}, $\tilde{f}$ factors uniquely through 

$$f:  \pi^{\textrm{alg}}(S, x)_{\textrm{ab}} \to  \pi^{\textrm{alg}}(S^{[n]}, \tilde{nx})$$

\section{Isomorphism of fundamental group schemes}\label{iso}

In this section, we show that $f$ is an isomorphism of affine group schemes. We begin by proving a result about $\mathfrak{S}_n$-equivariant stratified line bundles on $S^n$.

\begin{proposition}
Let $(L_i, \alpha_i)$ be a $\mathfrak{S}_n$-equivariant stratified line bundles on $S^n$. Then there exists a stratified line bundle $(\sL_i, \beta_i)$ such that $\rho^*(\sL_i, \beta_i) \simeq (L_i, \alpha_i)$
\end{proposition}

\begin{proof}
By Fogarty's result mentioned above, for any $\mathfrak{S}_n$-equivariant line bundle $L_i$ there exists line bundle $\sL_i \simeq \rho_*L_i^{\mathfrak{S}_n}$ such that $\rho^*\sL_i \simeq L_i$. Pushing forward $\alpha_i$ and taking $\mathfrak{S}_n$ invariants we get the isomorphism

$$\rho_*(F^*L_{i+1})^{\mathfrak{S}_n} \xrightarrow{\rho_*(\alpha_i)^{\mathfrak{S}_n}} \rho_*(L_i)^{\mathfrak{S}_n} $$

We show that the natural homomorphism 

$$F^*(\rho_*(L_{i})^{\mathfrak{S}_n}) \to (F^*\rho_*(L_{i})^{\mathfrak{S}_n}) \to (\rho_*F^*(L_{i})^{\mathfrak{S}_n}) $$  is an isomorphism. Pulling back the morphism under $\rho$, we get the commutative diagram

\begin{center}
\begin{tikzcd}
	{\rho^*F^*((\rho_*L_i)^{\mathfrak{S}_n})} &&& {\rho^*((\rho_*F^*L_i)^{\mathfrak{S}_n})} \\
	\\
	{F^*L_i} &&& {F^*L_i}
	\arrow[from=1-1, to=1-4]
	\arrow[from=1-1, to=3-1]
	\arrow["{=}", from=3-1, to=3-4]
	\arrow[from=1-4, to=3-4]
\end{tikzcd}
\end{center}

where the vertical morphisms are the natural morphism which are isomorphisms by Fogarty's theorem. By pushing forward under $\rho$ and taking $\mathfrak{S}_n$ invariants we get that 

$$F^*(\rho_*(L_{i})^{\mathfrak{S}_n}) \to (\rho_*F^*(L_{i})^{\mathfrak{S}_n}) $$ 

is an isomorphism. We define $\beta_i$ to be the composite isomorphism 

$$F^*(\rho_*(L_{i})^{\mathfrak{S}_n}) \to (\rho_*F^*(L_{i})^{\mathfrak{S}_n}) \xrightarrow{\rho_*(\alpha_i)^{\mathfrak{S}_n}} \rho_*(L_i)^{\mathfrak{S}_n} $$

The commutative diagram also gives us that $\rho^*(\sL_i, \beta_i) \simeq (L_i, \alpha_i)$
\end{proof}

\subsection{Faithfully flat}

Next we show that the morphism $f$ is faithfully flat

\begin{proposition}
The homomorphism $$f:  \pi^{\textrm{alg}}(S, x)_{\textrm{ab}} \to  \pi^{\textrm{alg}}(S^{[n]}, \tilde{nx})$$ is faithfully flat.
\end{proposition}

\begin{proof}
By [\cite{DM82} Theorem 2.21], this is equivalent to showing that the functor $$T: \mathcal{S}(S^{[n]}) \to \mathcal{S}(S^n)$$ is fully faithful and the essential image of $T$ is closed under taking subobjects. We already know that $T$ is fully faithful. Let $\sE_{\bullet} = (\sE_i, \alpha_i)$ be a stratified bundle on $S^{[n]}$ and  $\sF_{\bullet} := T(\sE_{\bullet})$ be the corresponding $\mathfrak{S}_n$-equivariant stratified bundle on $S^n$. If $\sF'_{\bullet} \subset \sF_{\bullet}$ is a $\mathfrak{S}_n$-equivariant stratified subbundle, then we need to show there exists $\sE'_{\bullet} \subset \sE_{\bullet}$ such that $T(\sE'_{\bullet}) = \sF'_{\bullet}$.
\\

The proof proceeds by induction on the rank of $\sE_{\bullet}$. If rank $\sE_{\bullet} = 1$, the proof is immediate. Let rank $\sE_{\bullet} \geq 2$

Then the stratified bundles $\sF_{\bullet}$ and $\sF'_{\bullet}$ correspond to the representations

$$\pi^{\textrm{alg}}(S^n, (x, \dots, x)\to \pi^{\textrm{alg}}(S, x)_{\textrm{ab}} \to GL(V)$$

and

$$\pi^{\textrm{alg}}(S^n, (x, \dots, x)\to \pi^{\textrm{alg}}(S, x)_{\textrm{ab}} \to GL(V')$$ respectively.
\\

As  $ \pi^{\textrm{alg}}(S, x)_{\textrm{ab}}$ is an abelian affine group scheme over $k$, all its irreducible representations are one dimensional. Thus one gets that the $ \pi^{\textrm{alg}}(S, x)_{\textrm{ab}}$-module $V/V'$ has a one dimensional quotient $W$. Thus there is a $ \pi^{\textrm{alg}}(S, x)_{\textrm{ab}}$-module surjection
$V \to W$ such that the kernel contains $V'$. Let $\sL_{\bullet}$ be the $\mathfrak{S}_n$-equivariant stratified bundle corresponding to $W$. Thus we have a short exact sequence of $\mathfrak{S}_n$-equivariant stratified bundles

$$0 \to \sK_{\bullet} \to \sF_{\bullet} \to \sL_{\bullet}\to 0$$

where $\sF'_{\bullet} \subset \sK_{\bullet}$.
\\

By proposition 1 above, we know that $L_i := \rho_*\sL_i^{\mathfrak{S}_n}$ is a line bundle on $S^{(n)}$ and $\rho^*L_i = \sL_i$

We claim that the following complex of sheaves on $S^{(n)}_*$ is exact for all $i \in \mathbb{N}$
\begin{equation}
\label{ses}
0 \to (\rho_*\sK_i)^{\mathfrak{S}_n}|_{S^{(n)}_*} \to (\rho_*\sF_i)^{\mathfrak{S}_n}|_{S^{(n)}_*} \to (\rho_*\sL_i)^{\mathfrak{S}_n}|_{S^{(n)}_*} \to 0
\end{equation}

It is enough to show that $(\rho_*\sF_i)^{\mathfrak{S}_n}|_{S^{(n)}_*} \to (\rho_*\sL_i)^{\mathfrak{S}_n}|_{S^{(n)}_*}$ is surjective. We note that $(\rho_*\sF_i)^{\mathfrak{S}_n}|_{S^{(n)}_*} = h_*(\sE_i|_{S^{[n]}_*})$. Let $C$ be the cokernel 

$$h_*(\sE_i|_{S^{[n]}_*}) \to (\rho_*\sL_i)^{\mathfrak{S}_n}|_{S^{(n)}_*} \to C \to 0$$

Pulling back under $\rho$, we get the following commutative diagram on $S^n_*$

\begin{center}
\begin{tikzcd}
	{\rho^*h_*(\mathcal{E}_i|_{S^{[n]}_*})} && {\rho^*((\rho_*\mathcal{L}_i)^{\mathfrak{S}_n}|_{S^{(n)}_*} )} && {\rho^*C} && 0 \\
	\\
	{\mathcal{F}_i} && {\mathcal{L}_i|_{S^n_*}} && 0
	\arrow[from=1-1, to=1-3]
	\arrow[from=1-3, to=1-5]
	\arrow[from=1-5, to=1-7]
	\arrow["{=}", from=1-3, to=3-3]
	\arrow["{=}", from=1-1, to=3-1]
	\arrow[from=3-1, to=3-3]
	\arrow[from=3-3, to=3-5]
\end{tikzcd}
\end{center}

The rows are exact and hence $\rho^*C = 0$. As $\rho$ is surjective, this implies $C = 0$. Thus $K_i :=  (\rho_*\sK_i)^{\mathfrak{S}_n}|_{S^{(n)}_*}$ is locally free on $S^{(n)}_*$
\\

Pulling back the exact sequence \eqref{ses} under $h$, we get a short exact sequence of locally free sheaves on $S^{[n]}_*$

$$0 \to h^*K_i|_{S^{[n]}_*} \to \sE_i|_{S^{[n]}_*} \to \tilde{L}_i|_{S^{[n]}_*} \to 0 $$

where $\tilde{L}_i := h^*L_i$.
\\

As the complement of $S^{[n]}_*$ in $S^{[n]}$ is of codimension $\geq 2$ and $\sE_i, L$ are locally free, the surjective morphism

$$\sE_i|_{S^{[n]}_*} \to \tilde{L}_i|_{S^{[n]}_*} $$

extends to a unique morphism $\tau_i: \sE_i \to \tilde{L}_i$. This is surjective as $L$ is of rank $1$ and $\tau := (\tau_i)$ give a nonzero morphism of stratified bundles $$\sE_{\bullet} \to \tilde{L}_{\bullet}$$

where $\tilde{L}_{\bullet} := h^*(\rho_*(\sL_{\bullet})^{\mathfrak{S}_n})$. Let $\kappa_{\bullet}$ be the kernel of the morphism $\sE_{\bullet} \to \tilde{L}_{\bullet}$. Then $T(\kappa_{\bullet}) = \sK_{\bullet}$. Thus, by the induction hypothesis on rank, there exists a stratified subbundle $\sE'_{\bullet} \subset \kappa_{\bullet} \subset \sE_{\bullet}$ such that $T(\sE'_{\bullet}) = \sF'_{\bullet}$.

\end{proof}

\subsection{Closed immersion}

We begin by recalling a result from \cite{PS20}.
\\

Let $p \in S^{(n)}$ be a point of type $(n_1, n_2, \dots, n_r)$. Let $p'_i$, for $i = 1, 2, \dots m$ be the points in the fiber $h^{-1}(p)$. Let $A$ be the local ring $\sO_{S^{(n)}, p}$ and $B$ be the semilocal ring $\sO_{S^n}\otimes_{\sO_{S^{(n)}}} A$. Then $B$ is a finite $A$ module and $B^{\mathfrak{S}_n} = A$.

\begin{lemma}\label{surj}
When char $k > n_1$, any $\mathfrak{S}_n$-equivariant surjective $B$-module homomorphism $f: M \to N$ of finitely generated $B$ modules descends to surjective $A$-module homomorphism of the $\mathfrak{S}_n$-invariants $M^{\mathfrak{S}_n} \to N^{\mathfrak{S}_n}$
\end{lemma} 

This allows us to prove the following analogue of Proposition 5.3.6 in \cite{PS20}.

\begin{proposition*}
Let $\sE_{\bullet} = (\sE_i, \alpha_i)$ be a $\mathfrak{S}_n$-equivariant stratified bundle on $S^n$

\begin{enumerate}
\item Let $p \in S^{(n)}$ be a point of type $(n_1, n_2, \dots, n_r)$. If char $k > n_1$, then the sheaf $\rho_*\sE_i^{\mathfrak{S}_n}$ is locally free in a neighbourhood of $p$ for all $i$.
\item Let $U$ denote the largest open subset where $\rho_*\sE_i^{\mathfrak{S}_n}$ is locally free, then on $\rho^{-1}(U)$, the natural morphism

$$\rho^*\rho_*\sE_i^{\mathfrak{S}_n} \to \sE_i$$

is an isomorphism for all $i \in \mathbb{N}$
\end{enumerate}
\end{proposition*}

\begin{proof}
The first assertion is proved by induction on the rank of $\sE_{\bullet}$. If $\sE_{\bullet}$ is a $\mathfrak{S}_n$-equivariant stratified bundle of rank $1$, then by proposition 1, $\rho_*\sE_i^{\mathfrak{S}_n}$ is locally free on $S^{(n)}$ for all $i$. In general, as $\sE_{\bullet}$ corresponds to a representation of the abelian group scheme $\pi^{\textrm{alg}}(S, x)_{\textrm{ab}}$, there exists a $\mathfrak{S}_n$-equivariant short exact sequence of locally free sheaves on $S^n$

$$0 \to \sK_{\bullet} \to \sE_{\bullet} \to \sL_{\bullet}\to 0$$

Pushing forward by $\rho$ and taking $\mathfrak{S}_n$-invariants we get the exact sequence for all $i$

$$0 \to \rho(\sK_i)^{\mathfrak{S}_n} \to \rho(\sE_i)^{\mathfrak{S}_n} \to \rho(\sL_i)^{\mathfrak{S}_n} $$

We claim that the homomorphism on the right is surjective in the neighbourhood of a point $p$ of type $(n_1, n_2, \dots, n_r)$. Surjectivity can be checked after passing to a formal neighbourhood of $p$ and thus reduces to lemma \ref{surj}. By induction hypothesis on rank, both $\rho(\sK_i)^{\mathfrak{S}_n}$ and $\rho(\sL_i)^{\mathfrak{S}_n}$ are locally free on a neighbourhood of $p$ and hence so is $\rho(\sE_i)^{\mathfrak{S}_n}$.
\\

The second assertion follows from the observation that the natural homomorphism 

$$\rho^*\rho_*\sE_i^{\mathfrak{S}_n} \to \sE_i$$

is an isomorphism on $\rho^{-1}(S^{(n)}_{\circ})$ as as $\rho: S^n_{\circ} \to S^{(n)}_{\circ}$ is finite \'etale. As the complement of $S^n_{\circ}$ in $\rho^{-1}(U)$ is of codimension $\geq 2$ and both sheaves are locally free on $\rho^{-1}(U)$, thus the natural morphism is an isomorphism.
\end{proof}

\begin{proposition}
Let char $k > 3$. The homomorphism $$f:  \pi^{\textrm{alg}}(S, x)_{\textrm{ab}} \to  \pi^{\textrm{alg}}(S^{[n]}, \tilde{nx})$$ is faithfully flat.
\end{proposition}

\begin{proof}
By [\cite{DM82}, Theorem 2.21], it is enough to show that the functor $$T: \mathcal{S}(S^{[n]}) \to \mathcal{S}(S^n)$$ is essentially surjective. Thus we want to show that for any $\mathfrak{S}_n$-equivariant stratified bundle $\sE_{\bullet}$ on $S^n$, there exists a stratified bundle $\sF_{\bullet}$ on $S^{[n]}$ such that $T(\sF_{\bullet}) = \sE_{\bullet}$.
\\

Let $U$ be the open subset of $S^{(n)}$ consisting of points of type $(1,1,\dots, 1), (2,1,\dots, 1), (3,1,\dots, 1)$ and $(2, 1,1,\dots, 1)$. By assumption on characteristic of $k$ and the previous proposition, we get that $\rho_*\sE_i^{\mathfrak{S}_n}$ is locally free on $U$. Also we have on $\rho^{-1}(U)$, the natural morphism

$$\rho^*\rho_*\sE_i^{\mathfrak{S}_n} \to \sE_i$$

is an isomorphism. Imitating proposition 1 above, this allows us to define a stratified bundle $(\rho_*\sE_i^{\mathfrak{S}_n}, \beta_i)$ on $U$ such that $\rho^*(\rho_*\sE_i^{\mathfrak{S}_n}, \beta_i) \simeq \sE_{\bullet}$. Pulling back under $h$ to $h^{-1}(U)$ (whose complement in $S^{[n]}$ has codimension $\geq 3$) and extending to $S^{[n]}$, we get a stratified bundle $\sF_{\bullet}$ such that $T(\sF_{\bullet}) = \sE_{\bullet}$ 
\end{proof}

As $f$ is both faithfully flat and a closed immersion, we get the following theorem 

\begin{theorem}
Let char $k > 3$. The homomorphism $$f:  \pi^{\textrm{alg}}(S, x)_{\textrm{ab}} \to  \pi^{\textrm{alg}}(S^{[n]}, \tilde{nx})$$ is an isomorphism.
\end{theorem}


\begin{thebibliography}{}
\bibitem[BHdS21]{BHdS21} Biswas, Indranil, Phùng Hô Hai, and João Pedro Dos Santos. "On the fundamental group schemes of certain quotient varieties." Tohoku Mathematical Journal 73, no. 4 (2021): 565-595.
\bibitem[BPS06]{BPS06} Biswas, Indranil, A. J. Parameswaran, and S. Subramanian. "Monodromy group for a strongly semistable principal bundle over a curve." Duke Mathematical Journal 132, no. 1 (2006): 1-48.
\bibitem[DM82]{DM82} Deligne, Pierre; Milne, James (1982), "Tannakian categories", in Deligne, Pierre; Milne, James; Ogus, Arthur; Shih, Kuang-yen (eds.), Hodge Cycles, Motives, and Shimura Varieties, Lecture Notes in Mathematics, vol. 900, Springer, pp. 101–228
\bibitem[dS07]{dS07} Dos Santos, João Pedro Pinto. "Fundamental group schemes for stratified sheaves." Journal of Algebra 317, no. 2 (2007): 691-713.
\bibitem[Fog68]{Fog68} Fogarty, John. "Algebraic families on an algebraic surface." American Journal of Mathematics 90, no. 2 (1968): 511-521.
\bibitem[Fog77]{Fog77} Fogarty, John. "Line bundles on quasi-symmetric powers of varieties." Journal of Algebra 44, no. 1 (1977): 169-180.
\bibitem[Gie75]{Gie75} Gieseker, David. "Flat vector bundles and the fundamental group in non-zero characteristics." Annali della Scuola Normale Superiore di Pisa-Classe di Scienze 2, no. 1 (1975): 1-31.
\bibitem[Har77]{Har77} Hartshorne, Robin. Algebraic geometry. Vol. 52. Springer Science \& Business Media, 2013.
\bibitem[Ish83]{Ish83} Ishimura, Sadao. "A descent problem of vector bundles and its applications." Journal of Mathematics of Kyoto University 23, no. 1 (1983): 73-83.
\bibitem[Lan11]{Lan11} Langer, Adrian. "On the S-fundamental group scheme." In Annales de l'Institut Fourier, vol. 61, no. 5, pp. 2077-2119. 2011.
\bibitem[Lan12]{Lan12} Langer, Adrian. "On the S-fundamental group scheme. II." Journal of the Institute of Mathematics of Jussieu 11, no. 4 (2012): 835-854.
\bibitem[N76]{N76} Nori, Madhav V. "On the representations of the fundamental group." Compositio Mathematica 33, no. 1 (1976): 29-41.
\bibitem[N82]{N82} Nori, Madhav V. "The fundamental group-scheme." Proceedings Mathematical Sciences 91, no. 2 (1982): 73-122.
\bibitem[PS20]{PS20} Paul, Arjun, and Ronnie Sebastian. "Fundamental group schemes of Hilbert scheme of n points on a smooth projective surface." Bulletin des Sciences Mathématiques 164 (2020): 102898.
\bibitem[SGA1]{SGA1} Grothendieck, Alexander, and Michele Raynaud. "Rev\^ etements\'etales et groupe fondamental (SGA 1)." arXiv preprint math/0206203 (2002).
\end{thebibliography}
\end{document}